\documentclass[12pt]{amsart}

\usepackage[utf8]{inputenc}
\usepackage[english]{babel}
\usepackage{amssymb,amsmath,amsthm,amsfonts,mathtools}
\usepackage{verbatim}
\usepackage{graphics}
\usepackage[margin=1in]{geometry}
\usepackage{amsthm}
\usepackage{amsmath}
\usepackage{amssymb}
\usepackage{relsize}
\usepackage{ stmaryrd }
\usepackage{todonotes}
\usepackage{hyperref}
\usepackage{cleveref}
\usepackage{enumitem}
\newtheorem{theorem}{Theorem}[section]
\newtheorem{lemma}[theorem]{Lemma}
\newtheorem{corollary}[theorem]{Corollary}
\newtheorem{proposition}[theorem]{Proposition}
\newtheorem{definition}[theorem]{Definition}
\usepackage{ wasysym }
\usepackage{color}

\usepackage{stackengine}
\stackMath

\newcommand{\A}{\mathbb{A}}
\newcommand{\Z}{\mathbb{Z}}

\newcommand{\F}{\mathbb{F}}
\newcommand{\K}{\mathbb{K}}

\newcommand{\f}{\textit{\textbf{f}}}

\DeclareMathOperator{\Span}{Span}

\newcommand{\Fq}{\F_q}

\title{Further Improvements to the Chevalley-Warning Theorems}
\author{David B. Leep}
\address{Department of Mathematics, University of Kentucky, Lexington, KY, 40506}
\email{leep@uky.edu}
\author{Rachel L. Petrik}
\address{Department of Mathematics, Rose-Hulman Institute of Technology, Terre Haute, IN, 47803}
\email{petrik@rose-hulman.edu}
\thanks{Acknowledgements here.}

\begin{document}
\begin{abstract}
    We study lower bound estimates for the number of solutions of systems of equations over finite fields. Heath-Brown improved the lower bounds given by the classical \emph{Chevalley-Warning Theorems} by excluding systems of equations whose solutions form an affine space. We improve each of Heath-Brown's results and demonstrate sharpness in several cases.
\end{abstract}

\maketitle

\section{Introduction}\label{intro}
The goal of this paper is to study lower bounds on the number of solutions to systems of polynomial equations over finite fields. Given a general system of polynomials, a classical question has been to determine the existence of a solution. 
For systems over finite fields, the key results on the existence of solutions are the \emph{Chevalley--Warning theorems} \cite{CW, Warning} (see \Cref{CW} below). In particular, Warning provided a lower bound in \cite{Warning} for the number of solutions. Easy examples exist to show that Warning's result is best possible. Avoiding these easy cases, Heath-Brown \cite{Heath-Brown} (see \Cref{thm: HB}) improved the lower bound given in \cite{Warning}. In this paper, we improve each of Heath-Brown's lower bounds (see \Cref{thm: HB Improvement}) by separating out the various cases more clearly and computing the estimates more precisely. In addition, we simplify the proof of \Cref{thm: HB} (2) for $q \geq 4$ and extend it to $q=3$. In \Cref{Sharpness}, we provide examples showing many of our bounds are optimal.

Let $\Fq$ denote the finite field of order $q = p^k$, where $p$ is a prime and $k \in \Z_{>0}$. We write $\f = \{f_1, \dots, f_r\}$ to denote a system of polynomials in $\Fq[x_1, \dots, x_n]$. Let $d_i = \deg(f_i)$ for $1 \leq i \leq r$. We define the degree of $\f$ to be $d \coloneqq d_1+ \dots + d_r$. Let $Z(\f, \Fq^n)$ be the set of zeros of $\f$ over $\Fq$ and let $N(\f, \Fq^n) = |Z(\f, \Fq^n)|$.

\begin{theorem}\label{CW} Let $f_{1}, \dots, f_{r} \in \Fq[x_{1},\dots, x_{n}]$, $\deg(f_{i}) = d_{i}$, $1 \leq i \leq r$. Assume $n > d_1 + d_2 + \cdots + d_r$. Then the following statements hold.
\begin{enumerate}
    \item \emph{(Chevalley, \cite{CW})} If $Z(\f, \Fq^n)$ is nonempty, then $N(\f,\Fq^{n}) \geq 2$.
    \item \emph{(Warning, \cite{Warning})} $N(\f, \Fq^{n}) \equiv 0 \bmod{p}$.
    \item \emph{(Warning, \cite{Warning})} If $Z(\f, \Fq^n)$ is nonempty, then $N(\f,\Fq^{n}) \geq q^{n-d}$.
    \item \emph{(Ax, \cite{Ax})} $N(\f, \Fq^{n}) \equiv 0 \bmod{q}$.
\end{enumerate}
\end{theorem}

A polynomial $f$ in $n$ variables is defined over $\Fq$ if $f \in \Fq[x_1, \dots, x_n]$. A polynomial $f \in \Fq[x_1, \dots, x_n]$ is a \emph{(homogeneous) form of degree} $d$, for some $d \geq 0$, if $f \neq 0$ and each monomial in $f$ has degree $d$. We say that $(a_1, \dots, a_n) \in \Fq^n$ is a zero of $\f$ if $f_i(a_1, \dots, a_n)=0$, $1 \leq i \leq r$. We say $\f$ is a \emph{homogeneous system} if each $f_i$ is a homogeneous form. If $\f$ is a homogeneous system with $\deg(f_i) \geq 1$ for each $i$, then $f_i(0, \dots , 0) = 0$ for each $i$. A zero $(a_1, \dots , a_n) \in \Fq^n$ of a homogeneous system $\f$ is a \emph{nontrivial zero} if $(a_1, \dots, a_n) \ne  (0, \dots , 0)$.

If $f_1, \dots, f_r$ are homogeneous forms and $n>d$, then the system has a nontrivial zero by Theorem \ref{CW} (1) and, therefore, there are at least $q^{n-d} \geq 2$ zeros by Theorem \ref{CW} (3). If $P$ is the number of projective zeros defined over $\Fq$, then $N(\f, \Fq^n)=P(q-1)+1$. This implies 
\begin{equation}\label{eq: N cong q-1}
    N(\f, \Fq^n) \equiv 1 \bmod q-1. 
\end{equation}

Further improvements require additional hypotheses, as examples exist (e.g. norm forms) that show the bound in Theorem \ref{CW} (3) is best possible. In these examples the set of zeros forms a subspace of dimension $n-d$ in $\Fq^n$. In fact, Heath-Brown showed (see Theorem \ref{thm: HB} below) that all examples where Theorem \ref{CW} (3) is sharp have the property that the set of zeros forms an affine space of $\Fq^n$.

\begin{definition}\label{def: Affine Space}
An \textit{affine space} is a coset of a subspace. We say that two affine spaces are \textit{parallel} if they are cosets of the same subspace. The \textit{dimension} of an affine space is the dimension of the corresponding subspace. We let $\A^t(\Fq)$ denote any affine space of dimension $t$ over $\Fq$.
\end{definition}

Since we work entirely in $\Fq^n$, it follows that $\A^n(\Fq) = \Fq^n$. When $Z(\f, \Fq^n)$ is not an affine space, Heath-Brown (\cite{Heath-Brown}) provided the following improvement to Theorem \ref{CW}. If $d=1$, then $Z(\f, \Fq^n)$ is always an affine space, which is excluded. Thus, without loss of generality, we can always assume $d \geq 2$.

\begin{theorem}[Heath--Brown, \cite{Heath-Brown}, Theorem 2]\label{thm: HB}
Suppose that $n>d$, $Z(\f, \Fq^{n})$ is non-empty, and $Z(\f, \Fq^n)$ is not an affine space of $\Fq^{n}$. Then the following statements hold. 
\begin{enumerate}
\item For any $q$, we have $N(\f,\Fq^{n}) > q^{n-d}$.
\item If $q \geq 4$, we have $N(\f,\Fq^{n}) \geq 2q^{n-d}$.
\item For any $q$, we have $N(\f,\Fq^{n}) \geq \dfrac{q^{n+1-d}}{(n+2-d)}$ provided that the polynomials $f_{1}, \dots, f_{r}$ are homogeneous.
\end{enumerate}
\end{theorem}

The main result of this paper, Theorem \ref{thm: HB Improvement}, is to show that each part of Theorem \ref{thm: HB} can be improved as follows. 

\begin{theorem}\label{thm: HB Improvement}
Suppose that $n>d$, $Z(\f, \Fq^{n})$ is non-empty, and $Z(\f, \Fq^n)$ is not an affine space of $\Fq^{n}$. Then the following statements hold.
\begin{enumerate}
\item For $q=2$, $N(\f,\F_2^{n}) \geq 2^{n-d} + 2$.
\item For $q \geq 3$, $N(\f,\Fq^{n}) \geq 2q^{n-d}$.
\item For $q \geq 2$, $N(\f,\Fq^{n}) > \dfrac{q^{n+1-d}-1}{q-1} \cdot \dfrac{q}{n+2-d}$ provided that the polynomials $f_{1}, \dots, f_{r}$ are homogeneous.
\item For $q \geq 3$, $N(\f, \Fq^n) \geq 2q^{n-d} + (q-2)q$ provided that the polynomials $f_1, \dots, f_r$ are homogeneous.
\end{enumerate}
Moreover, the bounds in (1), (2), and (4) are sharp under these hypotheses.
\end{theorem}

In Section \ref{Sharpness}, we state more precisely our results on the sharpness of these bounds. When $n >d$, the bound given in Theorem \ref{thm: HB Improvement} (3) is better than the result given in Theorem \ref{thm: HB} (3) because
\[
\dfrac{q^{n+1-d}-1}{q-1}\cdot \dfrac{q}{n+2-d} > q^{n-d} \cdot \dfrac{q}{n+2-d} = \dfrac{q^{n+1-d}}{n+2-d}.
\]

\section{Assumed Results from Heath-Brown's Paper}\label{Improvements to HB}
 
We need the following results, which are proven in \cite{Heath-Brown}, to prove Theorem \ref{thm: HB Improvement}.

\begin{theorem}[Heath--Brown, \cite{Heath-Brown}, Theorem 1]\label{thm: HB Theorem 1}
For any two parallel affine spaces $L_{1},L_{2} \subseteq \A^n(\Fq)$ of dimension at least d, we have
\[
N(\f, L_{1}) \equiv N(\f, L_{2}) \bmod{q}.
\]
\end{theorem}

\begin{lemma}[Heath--Brown, \cite{Heath-Brown}, Lemma 1] \label{lem: HB Lemma 1}
Let $L_{0} \subseteq \A^n(\Fq)$ be an affine space. Choose an affine space $L$ of maximal dimension $k$ such that $L \supseteq L_{0}$ and $N(\f, L) = N(\f, L_{0})$. Suppose $L' \supset L$ is an affine space of dimension $k+1$ such that $N(\f, L')$ is minimal.
Then $N(\f, L') > N(\f, L)$ and
\begin{equation*}\label{eq: HB Lemma 1}
N(\f, \A^n(\Fq)) \geq N(\f, L) + \dfrac{q^{n-k}-1}{q-1}(N(\f, L')-N(\f, L)).
\end{equation*}
\end{lemma}

For the next lemma, we require the following definition.

\begin{definition} Let $\K$ be a field. A set of points in $\K^t$ is in \textit{general position} if no $n$ of them lie in an $(n-2)$-dimensional affine space for $2 \leq n \leq t+1$.
\end{definition}

\begin{lemma}[Heath--Brown, \cite{Heath-Brown}, Lemma 2] \label{lem: HB Lemma 2}
Let $S \subseteq \A^t(\Fq)$ be a set containing $t+1$ points in general position. Then the following statements hold.
\begin{enumerate}
\item If $q=2$, and if there is no 2-dimensional affine space $L \subseteq \A^t(\Fq)$ meeting $S$ in exactly 3 points, then $S= \A^t(\Fq)$.
\item If $q \geq 3$, and if $l \subseteq S$ for every line $l$ meeting $S$ in at least two points, then $S= \A^t(\Fq)$.
\item If $q \geq 4$, and if $|S\cap l|$ $\geq q-1$ for every line $l$ meeting $S$ in at least two points, then $\A^t(\Fq)\setminus S$ is contained in a hyperplane.
\item If $m \geq 2$ is an integer, and if $|S\cap l|$ $\geq m+1$ for every line $l$ meeting $S$ in at least two points, then $|S|$ $\geq \dfrac{m^{t+1}-1}{m-1}$.
\end{enumerate}
\end{lemma}

Our proof of Theorem \ref{thm: HB Improvement} does not depend on Lemma \ref{lem: HB Lemma 2} (3), the most difficult part of Lemma \ref{lem: HB Lemma 2}, whereas the proof provided by Heath-Brown requires this statement.

Heath-Brown did not explicitly state the following lemmas, however, the proofs of these lemmas appear within the proof of his main result in \cite{Heath-Brown}. We include the proofs for the convenience of the reader.

\begin{lemma}[Heath--Brown, \cite{Heath-Brown}, page 431] \label{lem: HB Lemma 3} Let $L$ be an affine space of dimension $d$. If $n \geq d$ and $N(\f, L) = v$ where $1 \leq v \leq (q-1)$, 
then $N(\f, L^*) \ge v$ for every affine space $L^*$ of dimension $d$ parallel to $L$. 
As a result, $N(\f, \Fq^n) \geq vq^{n-d}$.
\end{lemma}

\begin{proof}
By \Cref{thm: HB Theorem 1}, $N(\f, L^*) \equiv N(\f, L) = v \bmod q$.
Since $1 \leq v \leq (q-1)$, we must have $N(\f, L^*) \ge v$.
Since there are $q^{n-d}$ affine spaces $L^*$ of dimension $d$ parallel to $L$, 
we have $N(\f, \Fq^n) \geq vq^{n-d}$.
\end{proof}

\begin{lemma}[Heath--Brown, \cite{Heath-Brown}, page 432]\label{lem: if dim=d}
Let $L \subseteq \Fq^n$ be an affine space. Suppose $2 \leq N(\f, L) \leq q-1$ and $\deg(\f)=d$. 
\begin{itemize}
    \item[(1)] Then $\dim(L) \leq d$.
    \item[(2)] If $\f$ is a homogeneous system, then $\dim(L) \leq d-1$.
\end{itemize}
\end{lemma}

\begin{proof}
\noindent \textit{(1)} If $\dim(L) >d$, then $q \mid N(\f, L)$ by Theorem \ref{CW} (4) and Theorem \ref{thm: HB Theorem 1}, which is a contradiction.

\noindent \textit{(2)} By (1), assume $\dim(L) = d$. First suppose $\boldsymbol{0} \in L$ and thus, $L$ is a subspace. Since $\f$ is a homogeneous system, we have $q-1 \mid N(\f,L)-1$ by Equation (\ref{eq: N cong q-1}). 
Thus, either $N(\f, L) = 1$ or $N(\f, L) \ge q$, which is a contradiction.

Now suppose $\boldsymbol{0} \notin L$. Let $L'$ be the $(d+1)$-dimensional subspace containing $L$ and $\boldsymbol{0}$. Since $\f$ is a homogeneous system, distinct zeros of $Z(\f,L)$ correspond bijectively to distinct projective zeros of $Z(\f,L')$. Therefore, $N(\f,L')=1+(q-1)N(\f,L)$ by Equation (\ref{eq: N cong q-1}). Theorem \ref{CW} (4) implies that $q \mid N(\f,L')$. 
Thus, $N(\f,L) \equiv 1 \bmod q$, and so either $N(\f, L) = 1$ or $N(\f, L) \ge q + 1$, which is a contradiction.
\end{proof}

\section{Proof of Theorem \ref{thm: HB Improvement}}\label{Improvements to HB}

The following inequality is stated without proof in a slightly weaker form (for $q \geq 2v+3$) in \cite{Heath-Brown}. Recall that $\lfloor x \rfloor$ denotes the greatest integer less than or equal to $x$.

\begin{lemma} \label{lem: HB Inequality}
Suppose $v \in \Z$. Let $v \geq 2$ and $q \geq 2v+1$, then $\Big \lfloor \dfrac{qv}{v+1} \Big \rfloor^v > \dfrac{q^v}{v+1}$ with the exception of $v=2$, $q=5, 7$.
\end{lemma}

\begin{proof}
First we demonstrate that it is sufficient to show 
\[
v+1 > \Big (1 + \dfrac{1}{q-1}\Big)^v \Big(1 + \dfrac{1}{v} \Big)^v. 
\]

Suppose
\[
v+1 > \Big (1 + \dfrac{1}{q-1}\Big)^v \Big(1 + \dfrac{1}{v} \Big)^v = \Big( \dfrac{q}{q-1}\Big)^v\Big( \dfrac{v+1}{v}\Big)^v. 
\]

Taking the $v^{\text{th}}$ root of both sides yields

\[
(v+1)^{\frac{1}{v}} > \Big(\dfrac{q}{q-1}\Big) \Big(\dfrac{v+1}{v}\Big).
\]

Rearranging these terms through multiplication, we find
\[
\dfrac{q}{(v+1)^{\frac{1}{v}}} < \dfrac{(q-1)v}{v+1} = \dfrac{qv}{v+1} - \dfrac{v}{v+1} \leq \Big \lfloor \dfrac{qv}{v+1} \Big \rfloor,
\]
which gives the result of the lemma after raising both sides to the $v^{\text{th}}$ power.

Assume $q \geq 2v+1$. Note that since $e^y>1+y$ for $y>0$, it follows that 
\begin{equation}\label{eq: e^x}
    e^x>\Big(1+\frac{x}{v} \Big)^v,
\end{equation}
for all $x>0$, $v>0$. Then
\[
1 + \dfrac{1}{q-1} \leq  1 + \dfrac{1}{2v} =1+\dfrac{\frac{1}{2}}{v}
\]
\[
\Big (1 + \dfrac{1}{q-1}\Big)^v \leq \Big(1 + \dfrac{\frac{1}{2}}{v}\Big)^v < e^{\frac{1}{2}}.
\]

Applying (\ref{eq: e^x}) 

\begin{align*}
\Big (1 + \dfrac{1}{q-1}\Big)^v \Big(1 + \dfrac{1}{v} \Big)^v < e^{\frac{1}{2}}e = e^{\frac{3}{2}} < 5 \leq v+1 \text{ for } v \geq 4.
\end{align*}

Now let $v=3$. Then 
\[
\Big (1 + \dfrac{1}{q-1}\Big)^v \Big(1 + \dfrac{1}{v} \Big)^v < e^{\frac{1}{2}}\Big(\dfrac{4}{3}\Big)^3 = e^{\frac{1}{2}}\dfrac{64}{27} < 3.91 < 4 = v+1.
\]

Now let $v=2$. Notice that
\begin{align*}
\Big(\dfrac{q}{q-1} \Big)^v \Big(1 + \dfrac{1}{v} \Big)^v &= \dfrac{9}{4}\Big(\dfrac{q}{q-1} \Big)^2. \\
\end{align*}

We want the values of $q$ satisfying the inequality
$\dfrac{9}{4}\Big(\dfrac{q}{q-1} \Big)^2 < 3.$ This inequality is satisfied if and only if $\Big(\dfrac{q}{q-1} \Big)^2 < \dfrac{4}{3}$, which is true if and only if $q \geq 8$.

\end{proof}

\begin{lemma}\label{lem: general position of zeros}
Assume $n>d$ and $Z(\f, \Fq^n)$ is nonempty. If $Z(\f, \Fq^n)$ is not an affine space of $\Fq^n$, then $Z(\f,\Fq^n)$ contains at least $n+2-d$ points in general position.
\end{lemma}

\begin{proof}
By Theorem \ref{CW} (3), $N(\f, \Fq^n) \geq q^{n-d}$. Suppose that a maximal set of points of $Z(\f, \Fq^n)$ in general position has at most $n+1-d$ points. Then this maximal set of points lies in some affine space $\A^{n-d}(\Fq)$. Then $Z(\f, \Fq^n) \subseteq \A^{n-d}(\Fq)$, for otherwise we would have a larger set of points in general position. Since $N(\f, \Fq^n) \ge q^{n-d}$ and $|\A^{n-d}(\Fq)| = q^{n-d}$, it follows that $N(\f, \Fq^n) = q^{n-d}$ and $Z(\f; \Fq^n)$ forms an affine space, which is a contradiction.
\end{proof}

We now begin the proof of Theorem \ref{thm: HB Improvement}. The proof of Theorem \ref{thm: HB Improvement} (1), (2), and (3) closely follows the ideas of Heath--Brown's proof in \cite{Heath-Brown}. 

\begin{proof}[Proof of Theorem 1.4] Since we must check many cases, we give a convenient labeling to easily track where we are in the proof. Case IIB2(b) means we are in the proof of (2), Case B, subcase 2, subsubcase b. \\

By Lemma \ref{lem: general position of zeros}, $Z(\f, \Fq^n)$ contains a maximal set of $t+1 \geq n+2-d$ points in general position. Then $n+1-d \leq t \leq n$. Let $\A^t(\Fq)$ be the $t$-dimensional affine space spanned by these $t+1$ points. Note that $|Z(\f, \A^t(\Fq))| \leq |Z(\f, \Fq^n)|$. \\


\noindent \textit{(1)} \textbf{Case I} Assume $q=2$ and let $S= Z(\f, \A^t(\F_2))$.

\textbf{Case IA} Suppose the hypotheses of Lemma \ref{lem: HB Lemma 2} (1) are satisfied. Then $Z(\f, \A^t(\F_2)) = \A^t(\F_2)$ and
\begin{align*}
N(\f, \F_2^n) &\geq N(\f,\A^t(\F_2)) = |Z(\f, \A^t(\F_2))| = |\A^t(\F_2)| =2^t \\ &\geq 2^{n+1-d} = 2^{n-d}+2^{n-d} \geq 2^{n-d}+2.
\end{align*}

\textbf{Case IB} Suppose the second hypothesis of Lemma \ref{lem: HB Lemma 2} (1) is not satisfied. Then there exists a 2-dimensional affine space $L_{0} \subseteq \A^t(\F_2)$ with $N(\f, L_{0}) = 3$. 

Take $L$ and $L'$ as in Lemma \ref{lem: HB Lemma 1}. Then $N(\f, L) = N(\f, L_{0}) = 3$. Note dim$(L) = k \leq d$, because if dim$(L) >d$, then $2 \mid N(\f, L)$ by Theorem \ref{CW} (2) and Theorem \ref{thm: HB Theorem 1}, which is a contradiction. Then Lemma \ref{lem: HB Lemma 1} yields the following bound
\begin{align*}
N(\f, \F_2^n) &\geq N(\f, L)+ \dfrac{2^{n-k}-1}{2-1}(N(\f, L')-N(\f, L)) \geq 3+ 2^{n-k}-1 \geq 2^{n-d} + 2.
\end{align*}

\indent \textbf{Remark}: If we use Theorem \ref{CW} (2) and Theorem \ref{thm: HB} (1), then the result follows more easily because $2^{n-d} + 2$ is the smallest integer greater than $2^{n-d}$ that is divisible by $2$. However, our intermediate estimates above give potentially better results. \\

\noindent \textit{(2)} \textbf{Case II} Assume $q \geq 3$ and let 
$S=Z(\f, \A^t(\Fq))$.

\textbf{Case IIA} Suppose the second hypothesis of Lemma \ref{lem: HB Lemma 2} (2) is satisfied. Then $Z(\f, \A^t(\F_q)) = \A^t(\F_q)$ and
\[
N(\f, \Fq^n) \geq N(\f, \A^t(\Fq)) = |Z(\f,\A^t(\Fq))|=|\A^t(\Fq)| =q^{t} \geq q^{n+1-d} > 2q^{n-d}.
\]

\textbf{Case IIB}  
Suppose the second hypothesis of Lemma \ref{lem: HB Lemma 2} (2) is not satisfied. 
Then there is a line $L_{0} \subseteq \A^t(\Fq)$ such that $2 \leq |L_0 \cap S|$ 
and $2 \leq N(\f, L_{0}) \leq q-1$. 
Take $L$ and $L'$ as in Lemma \ref{lem: HB Lemma 1}. 
Then $2 \leq N(\f, L) = N(\f, L_{0}) \leq q-1$, and so 
$\dim(L) = k \leq d$ by Lemma \ref{lem: if dim=d}.
Then there are three subcases.

\indent \indent \textbf{Case IIB(1)} Suppose $k \leq d-2$. By Lemma \ref{lem: HB Lemma 1},
\begin{align*}
N(\f, \Fq^n) &\geq N(\f, L)+ \dfrac{q^{n-k}-1}{q-1}(N(\f, L')-N(\f, L))\\ 
&> \dfrac{q^{n-k}-1}{q-1} \geq  \dfrac{q^{n+2-d}-1}{q-1} > q^{n+1-d} \geq 3q^{n-d} >2q^{n-d} .
\end{align*}

\indent \indent \textbf{Case IIB(2)} Suppose $k=d$. Then $2 \leq N(\f, L) \leq q-1$. 
Then by Lemma \ref{lem: HB Lemma 3}, we find $N(\f, \Fq^n) \geq 2q^{n-d}$.

\indent \indent \textbf{Case IIB(3)} Suppose $k = d-1$. Then either
\begin{itemize}[leftmargin=1in]
\item[(a)] $N(\f, L') - N(\f, L) \geq 2$ 
\item[(b)] $N(\f, L') - N(\f, L) =1$, and thus $3 \leq N(\f, L') \leq q$, or

\end{itemize}

\indent \indent \indent \textbf{Case IIB3(a)} Suppose (a) holds. 
Define $L'$ as in Lemma \ref{lem: HB Lemma 1}. 
Then Lemma \ref{lem: HB Lemma 1} gives the following bound.
\begin{align*}
N(\f, \Fq^n) &\geq N(\f, L)+ \dfrac{q^{n-k}-1}{q-1}(N(\f, L')-N(\f, L))\\ 
&> \dfrac{q^{n-k}-1}{q-1}(N(\f, L')-N(\f, L)) \geq 2 \cdot \dfrac{q^{n+1-d}-1}{q-1} > 2q^{n-d}.
\end{align*}

\indent \indent \indent \textbf{Case IIB3(b)} 
Suppose (b) holds. We have $\dim(L')=d$. First assume $2 \leq N(\f,L) \leq q-2$. 
Then $3 \leq N(\f,L') \leq q-1$. 
By Theorem \ref{thm: HB Theorem 1}, we have 
$N(\f, L^*) \equiv N(\f,L') \bmod{q}$ for every affine 
space $L^*$ of dimension $d$ parallel to $L'$. 
Then $N(\f, L^*) \notin \{0,1,2\}$ because $N(\f,L') \not \equiv 0,1,2 \bmod{q}$. 
Thus, $N(\f, L^*) \geq 3$. 
Since there are $q^{n-d}$ affine spaces $L^*$ of dimension $d$ that are parallel to $L'$, 
it follows that $N(\f, \Fq^n) \geq 3q^{n-d}>2q^{n-d}$.

Now assume $N(\f,L)=q-1$. Then $N(\f, L') = q$.
Let $v \in Z(\f, L)$, let $w \in Z(\f,L') \backslash L$, and let $L'_0$ 
be the line containing $v,w$.  Then $L'_0 \subset L'$ 
and $N(\f, L'_0) = 2$. Thus $|L'_0 \cap S| = 2$.
This means that we can assume from the start that $N(\f, L) = 2$.
If $q \ge 4$, then $N(\f, L) =2 \le q-2$ and we finish using the argument in the previous paragraph.

It remains to consider the case that $q = 3$.
For the first part of the following argument, it is more convenient to
consider general values of $q$.  At the end of the argument, we specialize 
to $q=3$.

Let $L''$ be an affine subspace of dimension $d+1$ containing $L$.
Since $N(\f,L')-N(\f,L)\ge 1$ for any affine space of dimension $d$ with
$L \subset L' \subset L''$, Lemma \ref{lem: HB Lemma 1} implies 
\begin{align*}
N(\f, L'') &\geq N(\f,L)+ \dfrac{q^{(d+1) - (d-1)}-1}{q-1}(N(\f,L')-N(\f,L))
\ge N(\f,L)+\dfrac{q^{2}-1}{q-1} \\
&= (q-1)+\dfrac{q^2-1}{q-1} = (q-1)+(q+1)=2q.
\end{align*}

Suppose that $N(\f, L'') \ge 2q+1$ for all affine spaces $L''$ having 
dimension $d+1$ and containing $L$.  
Then $N(\f, L'') \ge 3q$ for all such affine spaces $L''$ by 
Theorem \ref{CW} (4) and Theorem \ref{thm: HB Theorem 1}.
Since $N(\f,L'')-N(\f,L') \ge 3q - q = 2q$, Lemma \ref{lem: HB Lemma 1} implies that
\begin{align*}
N(\f, \Fq^n) &\geq N(\f,L')+ \dfrac{q^{n-d}-1}{q-1}(N(\f,L'')-N(\f,L'))
\ge q +\dfrac{q^{n-d}-1}{q-1} \cdot 2q \\
&= q + 2q (q^{n-d -1} + \cdots + q + 1) \ge q + 2q^{n-d} > 2q^{n-d}.
\end{align*}

Thus we can assume that there exists $L''$ as above such that $N(\f, L'') = 2q$.
Now assume that $q  = 3$, and so $N(\f, L'') = 6$.

Let $Z(\f, L'') = \{P_0, \ldots, P_5\}$.
We can assume that $Z(\f, L) = \{P_0, P_1\}$ and $Z(\f, L') = \{P_0, P_1, P_2\}$.
We can assume that $P_0 = (0, \ldots, 0)$ and $P_1 = (1,0,0, \ldots, 0)$.
Since $P_2 \notin L$, we can assume that $P_2 = (0,1,0, \ldots, 0)$.
Since $P_3 \notin L'$, we can assume that $P_3 = (0,0,1,0, \ldots, 0)$.
Since $P_4, P_5 \in L'' \backslash L'$, we can assume that
$P_4 = (a_1, a_2, \ldots, a_n)$ and $P_5 = (b_1, \ldots, b_n)$ where
some $a_i \ne 0$ with $i \ge 3$ and some $b_j \ne 0$ with $j \ge 3$.
We have $3 \le \dim(\Span\{P_1, \ldots, P_5\}) \le 5$.

We will find an affine space $\widehat{L}$ of dimension $n-1$ such that 
$N(\f, \widehat{L} \cap L'') \equiv 2 \bmod 3$. 
Then since $\widehat{L} \cap L''$ is nonempty and 
$N(\f, \widehat{L} \cap L'') \ne 6$, it follows that 
$\dim(\widehat{L} \cap L'') = (d+1) -1 = d$.
Then \Cref{lem: HB Lemma 3} implies that $N(\f, \F_q^n) \ge 2q^{n-d}$.

Suppose first that $\dim(\Span\{P_1, \ldots, P_5\}) = 5$.
Then we can assume that 
\[P_4 = (0,0,0,1,0, \ldots, 0) \text{ and } P_5 = (0,0,0,0,1,0, \ldots, 0).\]
Let $\widehat{L}: x_1 + x_2 + \cdots + x_5 = 1$, an affine space of dimension $n-1$.
Then $Z(\f, \widehat{L} \cap L'') = \{P_1, P_2, \ldots, P_5\}$.
Thus $N(\f, \widehat{L} \cap L'') = 5 \equiv 2 \bmod 3$.

Now suppose that $\dim(\Span\{P_1, \ldots, P_5\}) = 4$.
Then we can assume that 
\[P_4 = (0,0,0,1,0, \ldots, 0) \text{ and } P_5 = (b_1, b_2, b_3, b_4,0,\ldots,0).\]
Suppose that $b_i = 0$ for some $i \le 4$.
Let $\widehat{L}: x_i = 0$. 
Then $N(\f, \widehat{L} \cap L'') = 5 \equiv 2 \bmod 3$.
Now suppose that each $b_i \ne 0$.
Suppose that $b_i = 1$ for some $i \le 4$.
Let $\widehat{L}: x_i = 1$. 
Then $N(\f, \widehat{L} \cap L'') = 2$.
Now suppose that $b_1 = \cdots = b_4 = 2$.
Let $\widehat{L}: x_1 - x_2 + x_3 - x_4 = 0$. 
Then $N(\f, \widehat{L} \cap L'') = 2$.

Suppose that $\dim(\Span\{P_1, \ldots, P_5\}) = 3$.
Then we can assume that 
\[P_4 = (a_1, a_2, a_3,0, \ldots, 0) \text{ and } P_5 = (b_1, b_2, b_3, 0,\ldots,0).\]
Suppose that $a_i = 1$.  If $b_i \in \{0, 2\}$, then let
$\widehat{L}: x_i  = 1$, which gives $N(\f, \widehat{L} \cap L'') = 2$.
Thus we can assume that $b_i = 1$.
Suppose that $a_i = 0$.  If $b_i = 0$, then let $\widehat{L}: x_i  = 0$,
which gives $N(\f, \widehat{L} \cap L'') = 5$.
If $b_i = 1$, then let $\widehat{L}: x_i = 1$,
which gives $N(\f, \widehat{L} \cap L'') = 2$.
Thus we can assume that $b_i = 2$.
Suppose that $a_i = 2$.  If $b_i = 2$, then let $\widehat{L}: x_i = 2$,
which gives $N(\f, \widehat{L} \cap L'') = 2$.
If $b_i = 1$, then let $\widehat{L}: x_i = 1$, 
which gives $N(\f, \widehat{L} \cap L'') = 2$.
Thus we can assume that $b_i = 0$.

It follows that we can assume $b_i \equiv 2 - a_i \bmod 3$.
That is,
\[P_4 = (a_1, a_2, a_3, 0 \ldots, 0) \text{ and } 
P_5 = (2-a_1, 2-a_2, 2-a_3, 0, \ldots, 0).\]
Then $(a_1, a_2, a_3) \ne (0,0,0), (1,1,1), (2,2,2)$
because $P_4 \ne P_5$ and $P_4, P_5$ are distinct from $P_0$.

Suppose that $(a_1, a_2, a_3) = (0,1,2)$ or some permutation of these coordinates.
Then $(b_1, b_2, b_3) = (2,1,0)$ and we let $\widehat{L}:x_1 + x_3 = 2$. 
Then $N(\f, \widehat{L} \cap L'') = 2$.
Suppose that $(a_1, a_2, a_3) = (0,0,2)$.  Then $(b_1, b_2, b_3) = (2,2,0)$ and
we take $\widehat{L}: x_1 + x_3 = 2$.
Then $N(\f, \widehat{L} \cap L'') = 2$.
Suppose that $(a_1, a_2, a_3) = (1,1,0)$.  Then $(b_1, b_2, b_3) = (1,1,2)$ and
we take $\widehat{L}: x_1 + x_2 = 2$.
Then $N(\f, \widehat{L} \cap L'') = 2$.
Suppose that $(a_1, a_2, a_3) = (2,2,1)$.  
Then $P_5 = (b_1, b_2, b_3) = (0,0,1)= P_3$, which is a contradiction.
This covers all possible cases and the proof of Case II is complete.

\medskip
\noindent
\textbf{Remark:}  Some additional information can be obtained from the proof of Case IIB.
Suppose that $q = 3$, $\dim(L'') = d+1$, and $N(\f, L'') = 6$, as above.
Let $\widehat{L}$ be an affine space of dimension $n-1$ such that
$N(\f, \widehat{L} \cap L'') \equiv 2 \bmod 3$. 
Then, as noted above, $\dim(\widehat{L} \cap L'') = d$.
We can write
$L''= (\widehat{L} \cap L'') \cup (\widehat{L}_1 \cap L'') \cup (\widehat{L}_2 \cap L'')$
where $\widehat{L} \cap L'', \widehat{L}_1 \cap L'', \widehat{L}_2 \cap L''$
are parallel affine spaces each having dimension $d$, and thus the union is disjoint.
Then 
\[6 = N(\f, L'') = N(\f, \widehat{L} \cap L'') + N(\f, \widehat{L}_1 \cap L'') 
+ N(\f, \widehat{L}_2 \cap L'').\]
By \Cref{thm: HB Theorem 1}, each term on the right is congruent to 2 modulo 3 and
thus each term is at least $2$. Since the sum equals $6$, it follows that each term
equals $2$.  

As a consequence, it follows that $\dim(\Span\{P_1, \ldots, P_5\}) \ne 5$
because in that case above, we found an affine space $\widehat{L}$ such that 
$N(\f, \widehat{L} \cap L'') = 5$.

\noindent \textit{(3)} \textbf{Case III}
We begin with the case $q=2$. Since
\[\frac{2^{n+1-d} -1}{2^{n-d}} = 2 - \frac{1}{2^{n-d}} 
\le \begin{cases} \frac{3}{2} \le \frac{n+2-d}{2} & \text{ if $n-d=1$}\\
2 \le \frac{n+2-d}{2} & \text{ if $n-d \ge 2$,} \end{cases}\]
it follows that 
$\dfrac{2^{n+1-d}-1}{2-1}\cdot \dfrac{2}{n+2-d} \le 2^{n-d} < 2^{n-d} +2$,
which is the lower bound in Theorem \ref{thm: HB Improvement} (1).

Now assume $q \geq 3$. Let $v = n+1-d$, let $m = \Big\lfloor \dfrac{qv}{v+1} \Big\rfloor$, and let $S = Z(\f, \A^t(\Fq))$. Then $v \geq 2$ and $2 \leq m \leq q-1$. \\

\textbf{Case IIIA} Suppose the hypothesis of Lemma \ref{lem: HB Lemma 2} (4) is satisfied. Then $|Z(\f, \A^t(\F_q))| \geq \dfrac{m^{t+1}-1}{m-1}$. Thus,
\begin{align*}
N(\f, \Fq^n) &\geq N(\f, \A^t(\Fq)) \geq \dfrac{m^{t+1}-1}{m-1} \geq \dfrac{m^{n+2-d}-1}{m-1}.
\end{align*}

We will now show that $\dfrac{m^{n+2-d}-1}{m-1} > \dfrac{q^{n+1-d}-1}{q-1}\cdot \dfrac{q}{n+2-d}$. This is equivalent to showing
$$\dfrac{m^{v+1}-1}{m-1} > \dfrac{q^v-1}{q-1} \cdot \dfrac{q}{v+1}, $$
which is equivalent to showing 
$$ m^v+m^{v-1} + \dots + m + 1 > \dfrac{q^v+q^{v-1}+\dots+q}{v+1}.$$

First assume $v \geq 2$, $q \geq 2v+1$, and $(q,v) \neq (5,2), (7,2)$. Then by Lemma \ref{lem: HB Inequality}, $m^v > \dfrac{q^v}{v+1}$. Now observe that 
\[
m = \Big\lfloor \dfrac{qv}{v+1} \Big\rfloor \leq \dfrac{qv}{v+1} < q.
\]
Since $\dfrac{q}{m} >1$, we have for $1 \leq i \leq v$ that 
\[
m^i > \dfrac{q^{v-i}}{m^{v-i}} \cdot \dfrac{q^i}{v+1} \ge \dfrac{q^i}{v+1}.
\]
Therefore, $m^v+m^{v-1}+\dots+1 > \dfrac{q^v+q^{v-1}+ \dots +q}{v+1}$.

Suppose $(q,v)=(5,2)$. Since $\f$ is a homogeneous system, $N(\f, \F_{5}^n) \equiv 1 \bmod 4$ by Equation (\ref{eq: N cong q-1}). Since $n-d=v-1=1$, Theorem \ref{thm: HB Improvement} (2) implies $N(\f, \F_{5}^n) \geq 10$. Thus, 
\[
N(\f, \F_{5}^n) \geq 13 > 10 = \dfrac{5^2-1}{5-1} \cdot \dfrac{5}{3} = \dfrac{q^{n+1-d}-1}{q-1} \cdot \dfrac{q}{n+2-d}.
\]


Now suppose $(q,v)=(7,2)$. Since $\f$ is a homogeneous system, $N(\f, \F_{7}^n) \equiv 1 \bmod 6$ by Equation (\ref{eq: N cong q-1}). Since $n-d=1$, Theorem \ref{thm: HB Improvement} (2) implies $N(\f, \F_{7}^n) \geq 14$. Thus, 
\[
N(\f, \F_{7}^n) \geq 19 >\dfrac{56}{3} = \dfrac{7^2-1}{7-1} \cdot \dfrac{7}{3} = \dfrac{q^{n+1-d}-1}{q-1} \cdot \dfrac{q}{n+2-d}.
\] 
Thus, if $q \geq 2v+1$, then the proof is finished in this case.
\bigskip

\indent \textbf{Remark}: If we use Theorem \ref{thm: HB Improvement} (4), proved later, then we obtain a better bound in the cases $(q,v)=(5,2), (7,2)$. See the beginning of Section \ref{Comparing (3) and (4)}.
\bigskip

\indent \indent Now suppose $q \leq 2v$. Then $2v+2 \geq q+2$. Thus, 
\[
2 \geq \dfrac{q+2}{v+1} > \dfrac{q+1+\frac{1}{q-1}}{v+1} = \dfrac{q^2}{(v+1)(q-1)}.
\]

Multiply both sides by $q^{n-d}$, substitute $n+1-d$ for $v$, and apply 
Theorem \ref{thm: HB Improvement} (2) to obtain
\begin{align*}
    N(\f,\Fq^n) &\geq 2q^{n-d} > \dfrac{q^{n+2-d}}{(v+1)(q-1)} 
    > \dfrac{q^{n+2-d}-q}{(v+1)(q-1)} =  \dfrac{q^{n+1-d}-1}{q-1}\cdot \dfrac{q}{n+2-d}.
\end{align*}

\indent \textbf{Remark}: Heath-Brown's proof of this part omitted the case when $q \leq 2v$.  
The case when $q \le 2v$ appears again in \Cref{cor to comparison}.

\medskip

\textbf{Case IIIB} Suppose the hypothesis of Lemma \ref{lem: HB Lemma 2} (4) is not satisfied. Then there is a line $L_{0} \subseteq \A^t(\Fq)$ with $2 \leq N(\f, L_{0}) \leq m$. Take $L$ and $L'$ as in Lemma \ref{lem: HB Lemma 1}. Then $2 \leq N(\f, L) = N(\f, L_{0}) \leq m \leq q-1$, and so $\dim(L) = k \leq d-1$ by Lemma \ref{lem: if dim=d} (2). Then we have the following two subcases. 

\indent \indent \textbf{Case IIIB(1)} Suppose $k \leq d-2$. Then $n-k \ge n+2 - d$ and 
Lemma \ref{lem: HB Lemma 1} implies that
\begin{align*}
N(\f, \Fq^n) & \ge N(\f,L)+ \dfrac{q^{n+2-d}-1}{q-1}(N(\f, L')-N(\f,L)) > \dfrac{q^{n+2-d}-1}{q-1} \\ &>  \dfrac{q(q^{n+1-d}-1)}{q-1} >  \dfrac{q^{n+1-d}-1}{q-1}\cdot \dfrac{q}{n+2-d}.
\end{align*}

\indent \indent \textbf{Case IIIB(2)} Suppose $k = d-1$.
Then $\dim(L') = d$.  Since $N(\f, L') \ge N(\f,L) + 1 \ge 3$, 
Lemma \ref{lem: if dim=d} (2) implies that $N(\f,L') \ge q$.  
Since $N(\f,L) \le m$, we have $N(\f, L') - N(f,L) \ge q-m$.
Since $m = \Big\lfloor \dfrac{qv}{v+1} \Big\rfloor$,
Lemma \ref{lem: HB Lemma 1} implies that
\begin{align*}
N(\f, \Fq^n) &> \dfrac{q^{n-k}-1}{q-1}(N(\f, L')- N(\f, L) \geq \dfrac{q^{n+1-d}-1}{q-1}(q-m)\\ &\geq \dfrac{q^{n+1-d}-1}{q-1}\Big(q-\dfrac{qv}{v+1}\Big)  
= \dfrac{q^{n+1-d}-1}{q-1}\cdot \dfrac{q}{v+1} 
= \dfrac{q^{n+1-d}-1}{q-1}\cdot \dfrac{q}{n+2-d}.
\end{align*}

\noindent \textit{(4)} By Theorem \ref{thm: HB Improvement} (2), $N(\f, \Fq^n) \geq 2q^{n-d}$. By Theorem \ref{CW} (4), $N(\f, \Fq^n) = 2q^{n-d} + aq$ for some $a \in \Z_{\geq 0}$. Since $N(\f, \Fq^n) \equiv 1 \bmod q-1$, we know $2 + a \equiv 1 \bmod q-1$, which implies $a \equiv -1 \bmod q-1$. Since $a \in \Z_{\geq 0}$, we have $a \geq q-2$, which gives the desired result.\\

We will demonstrate that bounds (1), (2) and (4) in \Cref{thm: HB Improvement} 
are sharp under these hypotheses in Section \ref{Sharpness}, which will complete the proof.
\end{proof}

\section{Optimality of Theorem \ref{thm: HB Improvement}}\label{Sharpness}

\noindent \textit{Sharpness of Theorem \ref{thm: HB Improvement} (1)} We will show that Theorem \ref{thm: HB Improvement} (1) is optimal for every $d \geq 2$ and $n \in \{d+1,d+2\}$.

For each $d \in \Z$, $d \geq 1$, let $g_d \in \F_2[x_1,\dots,x_d]$ be a homogeneous form of degree $d$ having no nontrivial zero over $\F_2.$ For example, $g_d$ could be chosen to be a norm form.

\begin{proposition}
Let $d \in \Z$, $d \geq 2$, and let $n = d+1$. Then there exists a homogeneous form $f_d \in \F_2[x_1, \dots, x_n]$ of degree $d$ such that $|Z(f_d, \F_2^n)| = 4 = 2^{n-d}+2$ and $Z(f_d, \F_2^n)$ is not an affine space in $\F_2^n$.
\end{proposition}

\begin{proof}
For $d=2,$ let $f_2 = x_1x_2-x_3^2$. Then $|Z(f_2, \F_2^3)| = 4$ because if $x_1=0$, then the zeros over $\F_2$ are $(0,0,0)$ and $(0,1,0)$, while if $x_1=1$, then the zeros over $\F_2$ are $(1,0,0)$ and $(1,1,1)$. These four zeros do not form a subspace of $\F_2^3$ and thus do not form an affine space in $\F_2^3$.

Let $d \geq 3$ and suppose by induction that $f_{d-1} \in \F_2[x_1, \dots, x_d]$ is a homogeneous form of degree $d-1$ such that $|Z(f_{d-1},\F_2^d)|=4$ and $Z(f_{d-1},\F_2^d)$ is not an affine space in $\F_2^d$. Let $$f_d = x_{d+1}^d-f_{d-1}(x_1,\dots,x_d)x_{d+1}+g_d(x_1, \dots, x_d).$$
Then $f_d$ is a homogeneous form of degree $d$ in $\F_2[x_1,\dots,x_{d+1}]$. Let $a_1, \dots, a_d \in \F_2$. First suppose $f_{d-1}(a_1, \dots, a_d) \neq 0$. Then $(a_1, \dots, a_d) \neq 0$ and so $g_d(a_1, \dots, a_d) \neq 0$. Then for each $a_{d+1} \in \F_2$, $$f_d(a_1 \dots, a_{d+1})=a_{d+1}^d-a_{d+1}+g_d(a_1,\dots, a_d)=g_d(a_1, \dots, a_d) \neq 0.$$ Now suppose $f_{d-1}(a_1, \dots, a_d)=0.$ Then there exists a unique $a_{d+1}\in \F_2$ such that $a_{d+1}^d+g_d(a_1, \dots, a_d)=0.$ Thus, $|Z(f_d,\F_2^{n})|=|Z(f_d,\F_2^{d+1})|=|Z(f_{d-1},\F_2^d)|=4.$

Let $\pi : \F_2^{d+1} \rightarrow \F_2^d$ be the projection map defined by $\pi(a_1, \dots, a_{d+1}) = (a_1, \dots, a_d)$. Then $\pi(Z(f_d,\F_2^{d+1}))= Z(f_{d-1},\F_2^d).$ Since $Z(f_{d-1},\F_2^d)$ is not an affine space in $\F_2^d$, it follows that $Z(f_d,\F_2^{d+1})$ is not an affine space in $\F_2^{d+1}=\F_2^n.$
\end{proof}

\begin{proposition}
Let $d \in \Z$, $d \geq 2$, and let $n = d+2$. Then there exists a homogeneous form $f_d \in \F_2[x_1, \dots, x_n]$ of degree $d$ such that $|Z(f_d, \F_2^n)|=6=2^{n-d}+2$ and $Z(f_d, \F_2^n)$ is not an affine space in $\F_2^n$.
\end{proposition}

\begin{proof}
Any subset of $\F_2^n$ with cardinality $6$ is not an affine space in $\F_2^n$ because $6$ is not a $2$-power.

For $d=2,$ let $f_2=x_1^2+x_1x_2+x_2^2+x_3x_4.$ Then $|Z(f_2, \F_2^4)|=6$ because there are $4$ zeros in $\F_2^4$ if $x_3=1$ and there are $2$ zeros in $\F_2^4$ if $x_3=0$.

Let $d \geq 3$ and suppose by induction that $f_{d-1} \in \F_2[x_1, \dots,x_{d+1}]$ is a homogeneous form of degree $d-1$ with $|Z(f_{d-1},\F_2^{d+1})|=6.$ By an invertible linear change of variables, we can assume that $f_{d-1}(0,0,\dots,0,1)=0$.

Let $f_d=x_{d+2}^d-f_{d-1}(x_1, \dots,x_{d+1})x_{d+2}+g_d(x_1, \dots, x_d).$ Then $f_d$ is a homogeneous form of degree $d$ in $\F_2[x_1,\dots,x_{d+2}].$ Let $a_1, \dots, a_{d+1} \in \F_2$. First suppose $f_{d-1}(a_1, \dots, a_{d+1}) \neq 0$. If $a_1= \dots = a_d=0,$ then $f_{d-1}(0,0,\dots,0,a_{d+1}) = 0$ for $a_{d+1} \in \F_2$ because $f_{d-1}(0,\dots,0,1)=0$ by assumption. Thus $(a_1, \dots, a_d) \neq 0$ and so $g_d(a_1, \dots, a_d) \neq 0$. Then $$f_d(a_1, \dots, a_{d+1}, a_{d+2}) = a_{d+2}^d - a_{d+2}+g_d(a_1, \dots, a_d)= g_d(a_1, \dots, a_d) \neq 0.$$ 
Now suppose $f_{d-1}(a_1, \dots, a_{d+1})=0$. Then there exists a unique $a_{d+2} \in \F_2$ with $a_{d+2}^d + g_d(a_1, \dots, a_d)=0$. Thus, $|Z(f_d, \F_2^{n})|=|Z(f_d, \F_2^{d+2})|= |Z(f_{d-1},\F_2^{d+1})|=6$. 

\end{proof}

\noindent \textit{Sharpness of Theorem \ref{thm: HB Improvement} (2)} Consider the polynomial 
\[
g(x_1, x_2, x_3) = x_1x_2+x_3^2 +1
\]
over $\F_3$. Since $|Z(g,\F_3^3)|=6$ and $6$ is not a power of 3, $Z(g,\F_3^3)$ does not form an affine space in $\F_3^3$. The lower bound given in Theorem \ref{thm: HB Improvement} (2) is $2q^{n-d} = 2(3)^{3-2} = 6$. \\

\noindent \textit{Sharpness of Theorem \ref{thm: HB Improvement} (4)} 
An example with $d=2$, $n=3$ is given by the homogeneous form
\[h(x_1, x_2, x_3) = x_1x_2-x_3^2 \]
over $\F_q$ where $q \geq 3$. 
There are exactly $(q-1)q+q = q^2$ zeros of this equation. 
The lower bound given in Theorem \ref{thm: HB Improvement} (4) is 
$2q^{n-d} + (q-2)q = 2q + (q-2)q = q^2$. 
If $Z(h,\Fq^3)$ is an affine space, then $Z(h,\Fq^3)$ is a subspace since 
$\boldsymbol{0} \in Z(h,\Fq^3)$. 
This is not the case because $(1,0,0), (0,1,0) \in Z(h,\Fq^3)$, but, their sum, 
$(1,1,0) \notin Z(h,\Fq^3)$. 

A second example with $d=3$, $n=4$ is given by 
\[j(x_1, x_2, x_3, x_4) = x_1^3-x_2^2x_1+c(x_2,x_3,x_4),\] 
where $c(x_2,x_3,x_4) \in \F_3[x_2,x_3,x_4]$ is a cubic form having no nontrivial zero over $\F_3$, such as a norm form. Let $(a_1, a_2, a_3, a_4) \in \F_3^4$ be a zero of $j$. 
If $a_2 \neq 0$, then $a_1^3 - a_2^2a_1 = a_1^3 - a_1 = 0$, but 
$c(a_2,a_3,a_4) \neq 0$. Therefore, $j(a_1, a_2, a_3,a_4) \neq 0$. Thus, $a_2=0$. 
Let $a_3, a_4 \in \F_3$ be arbitrary. Then there exists a unique $a_1 \in \F_3$ such that 
$a_1^3 = -c(a_2,a_3,a_4)$. Therefore, $N(j,\F_3^4) = 9 = q^2 = 2q^{n-d} + (q-2)q$, 
which is the lower bound given in Theorem \ref{thm: HB Improvement} (4). 
If $Z(j, \F_3^4)$ is an affine space, then $Z(j, \F_3^4)$ is a subspace because 
$\boldsymbol{0} \in Z(j, \F_3^4)$. 
By the pigeonhole principle, there exist two distinct nonzero pairs 
$(a_3,a_4), (a_3',a_4') \in \F_3^2$ such that $a_1=a_1'$. 
Then $(a_1, 0, a_3, a_4) - (a_1',0, a_3', a_4') = (0,0,a_3-a_3', a_4-a_4')$. 
This is not a zero of $j$ because $c$ has no nontrivial zero over $\F_3$.

A third example with $d=2$ and $n=4$ is given by 
\[h(x_1, x_2, x_3, x_4) = x_1 x_2 + x_3^2 + x_4^2 \in \F_3[x_1,x_2,x_3,x_4].\]
Then $h$ has exactly $(q-1)q^2 + q = q^3 - q^2 + q = 21$ zeros over $\F_3$.
The lower bound given in Theorem \ref{thm: HB Improvement} (4) is 
$2q^{2} + (q-2)q = 3q^2 - 2q = 21$.
Since $|Z(h, \F_q^4)| = 21$ is not a power of $3$, it follows that 
$Z(h,\Fq^4)$ is not an affine space.

\section{Comparing (3) and (4) of Theorem \ref{thm: HB Improvement}}
\label{Comparing (3) and (4)}

We now compare the statements in Theorem 1.4 for the case of homogeneous forms.

We already observed at the beginning of the proof of Case III that when $q=2$ the bound in Theorem \ref{thm: HB Improvement} (3) is strictly less than the bound in Theorem \ref{thm: HB Improvement} (1).

Now assume that $q \ge 3$.  If $n - d = 1$, then
\[2q^{n-d} + (q-2)q = q^2 > \frac{q^2+q}{3} = 
\frac{q^{n+1-d}-1}{q-1} \cdot \dfrac{q}{n+2-d}.\]
In this case, Theorem \ref{thm: HB Improvement} (4) gives a larger lower bound than Theorem \ref{thm: HB Improvement} (3).

Next suppose $n - d \ge 2$. For real numbers $A$ and $B$, let
\def\stacktype{L} 
$A \  \mathbin{\stackanchor[11pt]{\stackanchor[7pt]{<}{=}}{>}} \  B$, denote the three statements $A<B$, $A=B$, $A>B$.


\begin{proposition}\label{prop: comparison}
Let $q \geq 3$, $n-d \geq 2$, and $v=n+1-d$.
\begin{enumerate}
    \item If $q > \Big(2 + \dfrac{q^2-2q}{q^{n-d}}\Big)v+1+\dfrac{q^2-2q}{q^{n-d}} - \Big(\dfrac{1}{q}+\dfrac{1}{q^2}+ \dots + \dfrac{1}{q^{n-d-1}}\Big)$, then Theorem \ref{thm: HB Improvement} (3) gives a larger lower bound than Theorem \ref{thm: HB Improvement} (4).
    \item If $q < \Big(2 + \dfrac{q^2-2q}{q^{n-d}}\Big)v+1+\dfrac{q^2-2q}{q^{n-d}} - \Big(\dfrac{1}{q}+\dfrac{1}{q^2}+ \dots + \dfrac{1}{q^{n-d-1}}\Big)$, then Theorem \ref{thm: HB Improvement} (4) gives a larger lower bound than Theorem \ref{thm: HB Improvement} (3).
\end{enumerate}
\end{proposition}

\begin{proof}
We have 
\begin{align*}
\dfrac{q^{n+1-d}-1}{q-1}\cdot\dfrac{q}{n+2-d} = \dfrac{q^{n+1-d}+q^{n-d}+\dots+q}{v+1} = \dfrac{q^{n-d}}{v+1}\Big(q+1+\dfrac{1}{q}+\dots+\dfrac{1}{q^{n-d-1}}\Big).
\end{align*}
We now begin the proof.
\begin{align*}
    q \  &\mathbin{\stackanchor[11pt]{\stackanchor[7pt]{<}{=}}{>}} \  \Big(2 + \dfrac{q^2-2q}{q^{n-d}}\Big)v+1+\dfrac{q^2-2q}{q^{n-d}} - \Big(\dfrac{1}{q}+\dfrac{1}{q^2}+ \dots + \dfrac{1}{q^{n-d-1}}\Big)
\end{align*}
if and only if
\begin{align*}
    q + 1 + \Big(\dfrac{1}{q}+\dfrac{1}{q^2}+ \dots + \dfrac{1}{q^{n-d-1}}\Big) \  &\mathbin{\stackanchor[11pt]{\stackanchor[7pt]{<}{=}}{>}} \  \Big(2 + \dfrac{q^2-2q}{q^{n-d}}\Big)(v+1) 
\end{align*}
if and only if
\begin{align*}
    \dfrac{q^{n-d}}{v+1} \Big(q + 1 + \dfrac{1}{q}+ \dots + \dfrac{1}{q^{n-d-1}}\Big) \  &\mathbin{\stackanchor[11pt]{\stackanchor[7pt]{<}{=}}{>}} \  q^{n-d}\Big(2 + \dfrac{q^2-2q}{q^{n-d}}\Big) = 2q^{n-d}+q^2-2q
\end{align*}
if and only if, using the first equation,
\begin{align*}
    \dfrac{q^{n+1-d}-1}{q-1}\cdot\dfrac{q}{n+2-d} \  &\mathbin{\stackanchor[11pt]{\stackanchor[7pt]{<}{=}}{>}} \ 2q^{n-d}+(q-2)q.
\end{align*}
\end{proof}


\begin{corollary}\label{cor to comparison} Let $q \geq 3$, $n-d \geq 2$, and $v=n+1-d$.
\begin{enumerate}
    \item[(1)] If $q \geq 3v+2$, then the hypothesis in Proposition \ref{prop: comparison} (1) holds.
    \item[(2)] If $q \leq 2v$, then the hypothesis in Proposition \ref{prop: comparison} (2) holds.
\end{enumerate}
\end{corollary}

\begin{proof}
(1) Since $n-d \ge 2$ and $q \ge 3$, we have $0 < \frac{q^2-2q}{q^{n-d}} < 1$, 
and so 
\[2 < \Big(2 + \dfrac{q^2-2q}{q^{n-d}}\Big) < 3.\]
This gives
\[3v + 2 > \Big(2 + \dfrac{q^2-2q}{q^{n-d}}\Big)v+1+\dfrac{q^2-2q}{q^{n-d}} - \Big(\dfrac{1}{q}+\dfrac{1}{q^2}+ \dots + \dfrac{1}{q^{n-d-1}}\Big).\]

\medskip
(2) 
We have $\Big(\dfrac{1}{q}+\dfrac{1}{q^2}+ \dots + \dfrac{1}{q^{n-d-1}}\Big) 
< \frac{1}{q-1} \le 1$. Thus 
\[1 - \Big(\dfrac{1}{q}+\dfrac{1}{q^2}+ \dots + \dfrac{1}{q^{n-d-1}}\Big) > 0.\]
This gives 
\[2v < \Big(2 + \dfrac{q^2-2q}{q^{n-d}}\Big)v+1+\dfrac{q^2-2q}{q^{n-d}} - \Big(\dfrac{1}{q}+\dfrac{1}{q^2}+ \dots + \dfrac{1}{q^{n-d-1}}\Big).\]
 
\end{proof}

\bibliographystyle{plain}
\bibliography{references}{}

\begin{thebibliography}{1}

\bibitem{Ax}
J.~Ax.
\newblock Zeroes of polynomials over finite fields.
\newblock {\em Amer. J. Math.}, 86:255--261, 1964.

\bibitem{CW}
C.~Chevalley.
\newblock D\'emonstration d'une hypoth\`ese de {M}. {A}rtin.
\newblock {\em Abh. Math. Sem. Univ. Hamburg}, 11(1):73--75, 1935.

\bibitem{Heath-Brown}
D.~R. Heath-Brown.
\newblock On {C}hevalley-{W}arning theorems.
\newblock {\em Uspekhi Mat. Nauk}, 66(2(398)):223--232, 2011.

\bibitem{Warning}
E.~Warning.
\newblock Bemerkung zur vorstehenden {A}rbeit von {H}errn {C}hevalley.
\newblock {\em Abh. Math. Sem. Univ. Hamburg}, 11(1):76--83, 1935.

\end{thebibliography}

\end{document}